\documentclass{amsart}

\usepackage{amssymb}
\usepackage{url} 

\newtheorem{theorem}{Theorem}

\newtheorem{lemma}[theorem]{Lemma}

\theoremstyle{definition}
\newtheorem{definition}[theorem]{Definition}
\newtheorem{definition and remark}[theorem]{Definition and Remark}
\newtheorem{remark and definition}[theorem]{Remark and Definition}

\newtheorem{remark}[theorem]{Remark}

\newtheorem{notation and remark}[theorem]{Notation and Remark}

\newcommand{\Gm}{{\mathfrak m}}

\newcommand{\Ga}{{\mathfrak a}}

\newcommand{\Gp}{{\mathfrak p}}
\newcommand{\longto}{\rightarrow}

\newcommand{\Max}{\text{\rm Max}}
\newcommand{\F}{\mathcal{F}}

\newcommand\Q{{\mathbb Q}}

\newfont{\bg}{cmr10 scaled\magstep4}
\newcommand{\bigzerol}{\smash{\hbox{\bg 0}}}
\newcommand{\bigzerou}{\smash{\lower1.7ex\hbox{\bg 0}}}

\begin{document}







\title[ ]
{The strong Lefschetz property for \\ complete intersections 
defined by \\ products of linear forms}


%

\author{Tadahito Harima}
\address{Niigata University \\ Department of Mathematics Education, Niigata, 950-2181 Japan}
\email{harima@ed.niigata-u.ac.jp}
\thanks{This work was supported by JSPS KAKENHI Grant Number 15K04812.}

\author{Akihito Wachi}
\address{Hokkaido University of Education \\  Department of Mathematics, Kushiro, 085-8580 Japan}
\email{wachi.akihito@k.hokkyodai.ac.jp}

\author{Junzo Watanabe}
\address{Tokai University \\ Department of Mathematics, Hiratsuka, 259-1292 Japan}
\email{watanabe.junzo@tokai-u.jp}


%



%

\begin{abstract}
We prove the strong Lefschetz property 
for certain complete intersections defined by products of linear forms, 
using a characterization of the strong Lefschetz property in terms of central simple modules. 
\end{abstract}

\maketitle




\section{introduction}

Let $K$ be a field 
and let $A = \bigoplus_{i=0}^c A_i$ be 
a standard graded Artinian $K$-algebra with $A_c \neq (0)$. 
We say that 
$A$ has the {\em strong Lefschetz property} (SLP) 
if there exists a linear form $z\in A_1$ 
such that the multiplication map 
$\times z^d: A_i \rightarrow A_{i+d}$ has full rank 
for all $1\leq d \leq c-1$ and $0\leq i \leq c-d$. 
We call $z\in A_1$ with this property a strong Lefschetz (SL) element. 
It is a long standing conjecture that 
every Artinian complete intersection has the SLP 
(cf.  \cite{HMMNWW} Conjecture 3.46,  \cite{HMNW} Theorem 2.3, 
\cite{HW0} Remark 20, \cite{I} Remark (P. 67), 
\cite{MN} Question 3.1, \cite{W} Example 3.9).

Let $\F$ be the family of Artinian complete intersections 
defined by products of linear forms. 
It seems noteworthy that for any member $A \in \F$, 
there exist $A', A'' \in \F$ such that 
$A$ can be expressed as an extension of  $A'$ and $A''$, 
namely we have the exact sequence:
\[
0 \rightarrow A' \rightarrow A \rightarrow A'' \rightarrow 0. 
\]
Abedelfatah \cite{A} used this fact inductively to prove that 
all members of $\F$ satisfy the conditions of the EGH conjecture. 

In this paper, suggested by Abedelfatah's proof, 
we prove that for $A \in \F$, 
if each linear factor in the generators is sufficiently general, 
the algebra $A$ has the SLP. 
Our tool is the central simple modules for Artinian algebras 
introduced in Harima-Watanabe \cite{HW1}. 
For reader's convenience 
we recall basic facts of the central simple modules in Section 2. 
The first result is Theorem~\ref{th1} in Section 3, 
which is a generalization of Theorem 1.2 in \cite{KMMW}. 
We give a new proof using a basic property of central simple modules. 
The second result is Theorem~\ref{th2} in Section 4. 
In this theorem, 
we prove the SLP for certain complete intersections 
defined by products of general linear forms. 
This is an answer to Problem 4.5 posed in \cite{KMMW}. 
It is possible to express the SLP of an Artinian algebra 
in terms of the Jordan type of the multiplication map by a general element. 
This is discussed in Section 4 
before we start proving Theorem~\ref{th2}. 

Throughout this paper, 
$K$ denotes a field of characteristic zero 
and $R=K[x_1,x_2,\ldots,x_n]$ denotes the polynomial ring 
over $K$ in $n$ variables with $\deg x_i = 1$.

\section{Central simple modules}

We recall the notion of central simple modules 
for a standard graded Artinian $K$-algebra 
and review a characterization of the SLP in terms of central simple modules. 
For details we refer the reader to \cite{HW0}, \cite{HW1}, \cite{HW2} and \cite{HW3}.

\begin{definition}
Let $A$ be a standard graded Artinian $K$-algebra 
and let $z$ be a linear form of $A$. 
Let $p$ be the least integer such that $0:z^p=A$. 
Then we have a descending chain of ideals in $A$: 
\[
A=(0:z^p)+(z) \supset (0:z^{p-1})+(z) \supset 
\cdots \supset (0:z)+(z) \supset (0:z^0)+(z)=(z). 
\]
From among the sequence of successive quotients 
\[
((0:z^{p-i})+(z))/((0:z^{p-i-1})+(z))
\]
for $0 \leq i \leq p-1$, 
pick the non-zero spaces and rename them 
\[
U_1, U_2, \ldots, U_s. 
\]
Note that $U_1=A/(0:z^{p-1})+(z)$. 
We call the graded $A$-module $U_i$ 
the {\em $i$-th central simple module of $(A,z)$}. 
\end{definition}

\begin{remark}\label{Re_csm} 
With the same notation as above, 
let $q$ be the least integer such that $((0:z^{q})+(z)) \neq (z)$. 
Then we have the following by the definition. 
\begin{itemize}
\item[(1)] 
$U_s=((0:z^{q})+(z))/(z)$. 
\item[(2)] 
If $(0:z^q)=A$, 
then $(A,z)$ has only one central simple module which is isomorphic to $A/zA$. 
\item[(3)] 
If $s>1$, 
then it is possible to regard $U_1,\ldots,U_{s-1}$ 
as the full set of the central simple modules of $(A/(0:z^q), \overline{z})$, 
where $\overline{z}$ is the image of $z$ in $A/(0:z^q)$. 
\end{itemize}
\end{remark}

\begin{definition}
Let $A$ be a standard graded Artinian $K$-algebra 
and let $V=\bigoplus_{i=a}^{b} V_i$ be a finite graded $A$-module with $V_a\neq (0)$ and $V_b\neq (0)$. 
Then, 
we say that $V$ has the SLP as an $A$-module 
if there exists a linear form $z$ of $A$ 
such that the multiplication map $\times z^d: V_i \rightarrow V_{i+d}$ 
has full rank for all $1\leq d\leq b-a$ and $a\leq i\leq b-d$.  
\end{definition}

\begin{theorem}[\cite{HW1} Theorem 1.2] \label{csm}
Let $A$ be a standard graded Artinian Gorenstein $K$-algebra. 
The following conditions are equivalent. 
\begin{itemize}
\item[(i)] 
$A$ has the SLP. 
\item[(ii)] 
There exists a linear form $z$ of $A$ such that 
all the central simple modules of $(A,z)$ have the SLP. 
\end{itemize}
\end{theorem}

\begin{remark}
Let $A$ be an Artinian Gorenstein $K$-algebra with the SLP. 
Then it is not always true that 
all central simple modules of $(A,z)$ have the SLP 
for all linear forms $z$ of $A$ 
(see Example 6.10 in \cite{HW1}).   
\end{remark}

\section{First theorem}

\begin{theorem}\label{th1}
Let $R=K[x_1,x_2,\ldots,x_n]$ be the polynomial ring, 
and let $d_1,$ $d_2, \ldots , d_n$ be positive integers.  
For $n\geq 2$, 
consider an Artinian complete intersection ideal $I$ of $R$ 
generated by the polynomials 
\[
\{ x_1^{d_1}l_1, x_2^{d_2}l_2,\ldots,x_n^{d_n}l_{n} \}, 
\]
where 
\[
\begin{array}{rcl}
l_i&=&a_{i1}x_1+a_{i2}x_2+\cdots+a_{ii+1}x_{i+1} \ \ (i=1,2,\ldots,n-1) \\
l_{n}&=&a_{n1}x_1+a_{n2}x_2+\cdots+a_{nn}x_n \\
\end{array}
\]
are linear forms of $R$. 
When $n=1$, put $I=(x_1^{d_1+1}) \subset R=K[x_1]$. 
Then $A=R/I$ has the SLP.  
\end{theorem}

\begin{remark}\label{diagonal} 
Let $M=[a_{ij}]$ be an $n\times n$ matrix with entries in $K$  
and let $d_1,d_2,\ldots,d_n$ be positive integers. 
Consider the ideal 
\[
I=I_{M,(d_1,\ldots,d_n)}=(x_i^{d_i}(\sum_{j=1}^n a_{ij}x_j) \mid 1\leq i\leq n)
\]
of $R=K[x_1,x_2,\ldots,x_n]$. 
Then, it follows from Lemma 2.1 in \cite{A} that 
$I$ is an Artinian complete intersection ideal 
if and only if all principal minors of $M$ are non-zero. 
In particular, 
if $I$ is an Artinian complete intersection ideal, 
then the diagonal  entries $a_{ii}$ are non-zero.  
\end{remark}

\begin{proof}[Proof of Theorem~\ref{th1}] 
We prove that 
all the central simple modules of $(A,z)$ have the SLP, 
where $z$ is the image of $x_1$ in $A$. 
Then this theorem follows by Theorem~\ref{csm}. 
We use induction on the number $n$ of variables. 
When $n=1$, 
we have that $(I:x_1^i)+(x_1)=(x_1)$ for $1\leq i \leq d_1$ 
and $(I:x_1^{d_1+1})+(x_1)=K[x_1]$. 
Hence, by Remark~\ref{Re_csm} (2), 
$(A, z)$ has only one central simple module which is isomorphic to $A/zA \cong K$, 
and the assertion is trivial.

So, let $n\geq 2$. 
Note that $a_{ii}\neq 0$ for all $i$ (see Remark~\ref{diagonal}).

Case 1: If $l_1=x_1$, 
then $(I:x_1^i)+(x_1)=I+(x_1)$ for $1\leq i \leq d_1$ and $(I:x_1^{d_1+1})+(x_1)=R$. 
Hence, $(A, z)$ has only one central simple module which is isomorphic to $A/zA$. 
We want to show that $A/zA$ has the SLP. 
Let $\overline{l_i}=l_i-a_{i1}x_1$ for $2 \leq i \leq n$. 
Then 
\[
A/zA \cong \overline{A} = K[x_2,x_3,\ldots,x_n]/(x_2^{d_2}\overline{l_2}, x_3^{d_3}\overline{l_3},\ldots,x_n^{d_n}\overline{l_n}). 
\]
By induction hypothesis $A/zA$ has the SLP.

Case 2: Consider the case $l_1=a_{11}x_1+a_{12}x_2$ $(a_{12} \neq 0)$. 
Let $U_1,U_2,\ldots,U_s$ be all the central simple modules of $(A,z)$. 
Since $(I:x_1^i)+(x_1)=I+(x_1)$ for $1\leq i < d_1$ 
and $(I:x_1^{d_1})+(x_1)=I+(l_1, x_1)$, 
it follows from Remark~\ref{Re_csm} (1) that 
the last central simple module $U_s$ is 
a principal ideal generated by the image $\overline{l_1}$ of $l_1$ in $\overline{A}$. 
Hence, noting that $a_{12}\neq 0$, 
we have that  
\[
U_s \cong \overline{A}/(0:\overline{l_1}) \cong B_1=K[x_2,x_3,\ldots,x_n]/(x_2^{d_2-1}\overline{l_2}, x_3^{d_3}\overline{l_3},\ldots,x_n^{d_n}\overline{l_n}).  
\]
Here inductively we may assume that $B_1$ has the SLP. 
Thus, $U_s$ has the SLP. 

Moreover, by Remark~\ref{Re_csm} (3), 
it is possible to regard $U_1,\ldots,U_{s-1}$ 
as the full set of the central simple modules of $(R/I+(l_1), \overline{z})$, 
where $\overline{z}$ is the image of $z$ in $R/I+(l_1)$. 
Let $\tilde{l_i}=l_i-a_{i1}x_1-\frac{a_{i1}a_{12}}{a_{11}}x_2$ for $2\leq i \leq n$, and let 
\[
B_2=K[x_2,x_3,\ldots,x_n]/(x_2^{d_2}\tilde{l_2}, x_3^{d_3}\tilde{l_3},\ldots,x_n^{d_n}\tilde{l_n}). 
\]
Then, since $R/I+(l_1) \cong B_2$ and $a_{12}\neq 0$, 
it follows that $U_1,\ldots,U_{s-1}$ is 
the central simple modules of $(B_2, \overline{x_2})$, 
where $\overline{x_2}$ is the image of $x_2$ in $B_2$. 
Here inductively we may assume that 
all the central simple modules of $(B_2, \overline{x_2})$ have the SLP, 
and hence $U_1,\ldots,U_{s-1}$ have the SLP. 
This completes the proof. 
\end{proof}

\section{Second theorem}

Let $\Q$ be the rational number field.

\begin{theorem}\label{th2}
Let $K$ be a field containing $\Q$ 
and let $R=K[x_1,x_2,\ldots,x_n]$ be the polynomial ring.  
Let $\{m_i\}$ and $\{d_{ij}\}$ be positive integers, 
and consider a graded ideal $I$ of $R$ generated by the polynomials as follows, 
\[
\{ F_1=\prod_{j=1}^{m_1}L_{1j}^{d_{1j}}, F_2=\prod_{j=1}^{m_2}L_{2j}^{d_{2j}}, \ldots, F_n=\prod_{j=1}^{m_n}L_{nj}^{d_{nj}} \}, 
\]
where 
\[
L_{ij}=\xi_1^{(i,j)}x_1+\xi_2^{(i,j)}x_2+\cdots+\xi_n^{(i,j)}x_n \ \ (\xi_k^{(i,j)} \in K)
\]
are linear forms for all $i$ and $j$. 
Assume that the coefficients 
\[
\{ \xi_k^{(i,j)} \mid 1\leq i \leq n, 1 \leq j \leq m_i, 1\leq k \leq n \}
\]
of $L_{ij}$ are algebraically independent over $\Q$. 
Then we have the following. 
\begin{itemize} 
\item[(1)] 
$I$ is an Artinian complete intersection of $R$.  
\item[(2)] 
$A=R/I$ has the SLP. 
\end{itemize}
\end{theorem}

To prove Theorem~\ref{th2}, 
we review a characterization of SL elements in terms of Jordan types, 
which is explained in Section 3.5 of \cite{HMMNWW}.

Let $(A,\Gm)$ be an Artinian local ring 
which contains the residue field $K=A/\Gm$, 
and let $y\in \Gm$. 
Consider the map $\times y: A\longto A$ defined by $a\mapsto ya$. 
Since $A$ is Artinian, 
the linear map $\times y$ is nilpotent,  
and hence the Jordan canonical matrix 
$J_A(\times y)$ of $\times y$ is a matrix of the  following form: 
\[
J_A(\times y)=
\begin{bmatrix}
J(n_1) &     &        & \bigzerou        \\ 
    & J(n_2) &        &                   \\
    &     & \ddots &                        \\
\bigzerol   &      &        & J(n_r)      \\     
\end{bmatrix},  
\]
where $J(m)$ is the Jordan block of size $m$ 
\[
J(m)=
\begin{bmatrix}
0 &  1   &        & \bigzerou       \\
    & 0 &  \ddots      &            \\
    &     & \ddots &    1           \\
\bigzerol   &      &        & 0     \\     
\end{bmatrix}.  
\]
We denote the {\em Jordan type} of $\times y$ by writing 
\[
P_A(\times y)=n_1\oplus n_2\oplus \cdots \oplus n_r. 
\]
We note that 
\[
r=\dim_K A/yA \ \text{ and } \ \dim_K A=\sum_{i=1}^r n_i. 
\]
Two types 
$n_1\oplus n_2\oplus \cdots \oplus n_r$ 
and $n_1'\oplus n_2'\oplus \cdots \oplus n_{r'}'$ 
are regarded  as the same 
if they differ only by permutation of blocks. 
Thus they are treated as partitions of a positive integer.

Let $P = n_1 \oplus n_2 \oplus \cdots \oplus n_r $ and
$Q =m_1 \oplus m_2 \oplus \cdots \oplus m_{s}$ 
be two partitions of a positive integer $n$, 
i.e., 
all $n_i$ and $m_j$ are positive integers 
satisfying $\sum_{i=1}^rn_i=\sum_{j=1}^sm_j=n$, 
where we assume that 
$n_1 \geq n_2 \geq \cdots \geq n_r$ 
and $m_1 \geq m_2 \geq \cdots  \geq m_{s}$.
Then we write $P \succ Q$ 
if and only if (i) $r < s$ or (ii) $r=s$ and
$n_i=m_i$ for $i=1,2, \cdots, k-1$ and $n_k > m_k$. 
So $\succ$ is a total order on the set of all partitions of $n$.

For a partition $P=n_1\oplus n_2\oplus\cdots\oplus n_r$ 
of a positive integer $n$, 
define the polynomial 
\[
Y(P,\lambda)=\sum_{i=1}^r(1+\lambda+\lambda^2+\cdots+\lambda^{n_i-1}), 
\]
and let $d_i$ be the coefficient of $\lambda^i$ in $Y(P,\lambda)$ 
for all $i=0,1,\ldots,c$, 
where $c=\Max_i \{n_i-1\}$. 
Then we call $\widehat{P}=d_0\oplus d_1\oplus\cdots\oplus d_c$ 
the {\em dual partition} of $P$.

\begin{lemma}[\cite{HMMNWW} Proposition 3.64]\label{SL element}
Let $A=\bigoplus_{i=0}^c A_i$ be a standard graded Artinian $K$-algebra 
with $A_c \neq (0)$, and let $y\in\bigoplus_{i=1}^cA_i$ be a homogeneous element. 
Put $h_i=\dim_KA_i$ for all $0\leq i\leq c$. 
\begin{itemize}
\item[(1)] 
Suppose that the Hilbert function $\{h_i\}$ of $A$ is unimodal. 
Then $P_A(\times y)$ is less than or equal to the dual partition of 
$h_0 \oplus h_1 \oplus \cdots \oplus h_c$ 
with respect to the total order $\succ$. 
\item[(2)] 
Suppose that $y\in A_1$ is a linear form. 
Then the following conditions are equivalent.  
\begin{itemize}
\item[(i)] 
$y$ is an SL element of $A$. 
\item[(ii)] 
the Hilbert function of $A$ is unimodal 
and 
$P_A(\times y)$ is equal to the dual partition of 
$h_0 \oplus h_1 \oplus \cdots \oplus h_c$. 
\end{itemize}
\end{itemize}
\end{lemma}

The following remark is needed in the proof of Lemmas below.

\begin{remark}\label{Re of JD} 
Let $K$ be a field and let $K'$ be a subfield of $K$. 
Let $B$ be a standard graded Artinian $K'$-algebra, 
and put $A=B\otimes_{K'}K$. 
Then it is easy to see that 
$P_B(\times y)=P_{A}(\times (y\otimes 1))$ for all linear forms $y$ in $B$. 
Since $A$ and $B$ have the same Hilbert function, 
it follows from Lemma~\ref{SL element} (2) that 
if $y$ is an SL element of $B$ 
then $y\otimes 1$ is also an SL element of $A$. 
Thus, if $B$ has the SLP then $A$ also has the SLP. 
\end{remark}

\begin{lemma}\label{key lemma}
Let  $R=K[x_1,x_2, \ldots,x_n]$ be the polynomial ring over a field $K$. 
Let $K'$ be a subfield of $K$ 
and let $J$ be an Artinian graded ideal of $R'=K'[x_1, x_2, \ldots, x_n]$. 
Assume that 
the elements $\{\eta_1,\eta_2,\ldots,\eta_n\}$ in $K$ 
are algebraically independent over $K'$. 
Assume that $B=R'/J$ has the SLP. 
Then the image of $L=\eta_1x_1+\eta_2x_2+\cdots+\eta_nx_n$ in $A=R/JR$ is 
an SL element of $A$. 
\end{lemma}

\begin{proof} 
Let $X_1,X_2,\ldots,X_n$ be indeterminates over $B$,  
and let $K'(X)=K'(X_1, $ 
$X_2,\ldots,X_n)$ be the rational function field over $K'$. 
Let $B(X)$ denote the polynomial ring $B[X_1,X_2,\ldots,X_n]$ 
localized at the minimal prime ideal $\Gm B[X_1, $ 
$X_2, \ldots, X_n]$, 
where $\Gm$ is the unique maximal graded ideal of $B$, 
and put $Y=X_1y_1+X_2y_2+ \cdots +X_ny_n \in B(X)$, 
where $y_i$ is the image of $x_i$ in $B$ for all $1\leq i\leq n$. 
Then it is easy to see that $B(X) \cong B \otimes_{K'} K'(X)$. 

Let $y$ be an SL element of $B$ 
and let $P_B(\times y)$ be the Jordan type of $\times y$. 
Then $P_B(\times y)$ is the dual partition of $h_0\oplus h_1\oplus\cdots\oplus h_c$ 
by Lemma~\ref{SL element} (2), 
where $\{h_i\}$ is the Hilbert function of $B$. 
Hence it follows 
that $P_B(\times y)=P_{B(X)}(\times Y)$ 
from lemma~\ref{SL element} and Theorem 5.6 in \cite{HMMNWW}. 

Let $K'(\eta)$ be the subfield of $K$ 
generated by $\{\eta_1,\eta_2,\ldots,\eta_n\}$ over $K'$.  
Put $B(\eta)=B\otimes_{K'}K'(\eta)$ 
and $l=\eta_1y_1+\eta_2y_2+\cdots+\eta_ny_n \in B(\eta)$. 
Then, noting that $B(X)=B\otimes_{K'}K'(X)$, 
we have $B(X) \cong B(\eta)$, and 
\[
\dim_{K'(X)} ((0:_{B(X)}Y^{i+1})/(0:_{B(X)}Y^i)) 
= \dim_{K'(\eta)} ((0:_{B(\eta)}l^{i+1})/(0:_{B(\eta)}l^i))
\]
for all $i$. 
Thus $P_{B(\eta)}(\times l)=P_{B(X)}(\times Y)$ 
by Proposition 3.60 in \cite{HMMNWW}, 
and hence $P_{B(\eta)}(\times l)=P_{B}(\times y)$. 
Since $B(\eta)$ and $B$ have the same unimodal Hilbert function, 
$P_{B(\eta)}(\times l)$ is equal to the dual partition 
of the Hilbert function of $B(\eta)$. 
Therefore $l$ is an SL element of $B(\eta)$ by Lemma~\ref{SL element} (2), 
and $B(\eta)$ has the SLP. 
Thus, since $A=B(\eta)\otimes_{K'(\eta)}K$, 
it follows from Remark~\ref{Re of JD} 
that the image of $L$ in A is an SL element of $A$, 
and $A$ also has the SLP. 
\end{proof}

\begin{lemma}\label{second key lemma}
Let $R=K[x_1,x_2, \ldots,x_n]$ be the polynomial ring over a field $K$. 
Let $K'$ be a subfield of $K$ 
and let $J$ be an Artinian graded ideal of $R'=K'[x_1, x_2, \ldots, x_n]$ 
whose quotient algebra has the SLP. 
Assume that 
the elements $\{\xi_{ij}\}$ in $K$ 
are algebraically independent over $K'$, 
and put $L_i=\xi_{i1}x_1+\xi_{i2}x_2+ \cdots +\xi_{in}x_n \in R$ for all $1\leq i\leq m$.  
Then $A=R/(JR:L_1^{d_1}L_2^{d_2} \cdots L_m^{d_m})$ has the SLP 
for all positive integers $d_1,d_2,\ldots,d_m$.  
\end{lemma}

\begin{proof}
Put $K^{(0)}=K'$ 
and let $K^{(i)}$ be the subfield of $K$ 
generated by $\{\xi_{i1},\xi_{i2},$ $\ldots,\xi_{in}\}$ over $K^{(i-1)}$ 
for $1\leq i\leq m$. 
Put $R^{(i)}=K^{(i)}[x_1,x_2,\ldots,x_n]$ for $0\leq i\leq m$, 
$J^{(0)}=J$ 
and $J^{(i)}=J^{(i-1)}R^{(i)}:L_i^{d_i}$ for $1\leq i\leq m$. 
We prove by induction on $i=0,1,\ldots,m$ that 
$R^{(i)}/J^{(i)}$ has the SLP for $0\leq i \leq m$. 
The first case $i=0$ follows from our assumption. 

Assume that $R^{(i)}/J^{(i)}$ has the SLP. 
Then, it follows that 
the image of $L_{i+1}$ in $R^{(i+1)}/J^{(i)}R^{(i+1)}$ is 
an SL element of $R^{(i+1)}/J^{(i)}R^{(i+1)}$ by Lemma~\ref{key lemma}, 
and hence $R^{(i+1)}/J^{(i+1)}$ has the SLP by Proposition 3.11 in \cite{HMMNWW}. 
Therefore 
\[
R^{(m)}/J^{(m)} \cong B=R^{(m)}/(JR^{(m)}:L_1^{d_1}L_2^{d_2} \cdots L_m^{d_m})
\]
has the SLP by induction. 
Thus, since $A \cong B \otimes_{K^{(m)}} K$, 
we have that $A$ has the SLP by Remark~\ref{Re of JD}. 
\end{proof}

\begin{proof}[Proof of Theorem~\ref{th2}] 
(1) Let $\Gp$ be a prime ideal of $R$ containing $I$. 
Then it follows that $\Gp \supset \Ga=(L_{1j_1}, L_{2j_2},\ldots,L_{nj_n})$ 
for some $j_1,j_2,\ldots,j_n$. 
Let $M=[m_{ik}]$ be the $n\times n$ matrix 
consisting of coefficients of $L_{1j_1},L_{2j_2},\ldots,L_{nj_n}$, 
that is, $m_{ik}=\xi_k^{(i,j_i)}$ is the coefficient of $x_k$ in $L_{ij_i}$. 
Noting that 
the elements $\{\xi_k^{(i,j_i)}\}$ are algebraically independent over $\Q$, 
we have $\det M\neq 0$, 
hence $\Gp=\Ga$ is a maximal ideal of $R$. 
This shows that $I$ is an Artinian complete intersection ideal of $R$. 

(2) 
Let $z$ be the image of $L_{11}$ in $A$. 
By Theorem~\ref{csm}, 
it suffices to prove that all the central simple modules of $(A,z)$ have the SLP. 
We use induction on the number $n$ of variables. 
The first case $n=1$ is obvious. 

So we consider the case $n>1$. 
Here we use induction on the number $m_1$ of linear forms $L_{1j}$. 
Let $m_1=1$. 
Then it follows by Lemma 6.1 in \cite{HW1} that 
$(A,z)$ has only one central simple module which is isomorphic to $A/zA$. 
We want to show that $A/zA$ has the SLP. 
Put 
\[
\overline{L}_{ij}=L_{ij}-\xi_1^{(i,j)}x_1-
\frac{\xi_1^{(i,j)}}{\xi_1^{(1,1)}}\sum_{k=2}^n\xi_k^{(1,1)}x_k
=\sum_{k=2}^n
(\xi_k^{(i,j)}-\frac{\xi_1^{(i,j)}\xi_k^{(1,1)}}{\xi_1^{(1,1)}})x_k 
\] 
for all $(i,j)\neq (1,1)$, 
and note that the coefficients 
\[
\Xi_1=\{ \xi_k^{(i,j)}-\frac{\xi_1^{(i,j)}\xi_k^{(1,1)}}{\xi_1^{(1,1)}} \mid 
2\leq i \leq n, 1 \leq j \leq m_i, 2\leq k\leq n \}
\] 
of $\overline{L}_{ij}$ $(2\leq i\leq n, 1\leq j\leq m_i)$ 
are algebraically independent over $\Q$. 
Hence, since 
\[
A/zA \cong \overline{A} = K[x_2,x_3,\ldots,x_n]/
(\prod_{j=1}^{m_2}\overline{L}_{2j}^{d_{2j}}, \ldots, \prod_{j=1}^{m_n}\overline{L}_{nj}^{d_{nj}}) 
\] 
and $\overline{A}$ has the SLP by the induction hypothesis on $n$, 
we have that $A/zA$ has the SLP.

Let $m_1>1$. 
Let $U_1,U_2,\ldots,U_s$ be all the central simple modules of $(A,z)$. 
Since $(I:L_{11}^i)+(L_{11})=I+(L_{11})$ for all $i < d_{11}$ 
and $(I:L_{11}^{d_{11}})+(L_{11})=I+(\prod_{j=2}^{m_1}L_{1j}^{d_{1j}}, L_{11})$, 
we have that the last central simple module $U_s$ is 
a principal ideal generated 
by the image of $\prod_{j=2}^{m_1}L_{1j}^{d_{1j}}$ in $\overline{A}$. 
Let $l_{1j}$ be the image of $\overline{L}_{1j}$ in $\overline{A}$ for all $2\leq j\leq m_1$, 
and let $K_1$ be the subfield of $K$ generated by the set $\Xi_1$ defined above. 
Note that the coefficients 
\[
\Xi_2=\{ \xi_k^{(1,j)}-\frac{\xi_1^{(1,j)}\xi_k^{(1,1)}}{\xi_1^{(1,1)}} 
\mid 2\leq j \leq m_1, 2\leq k\leq n \}
\] 
of $\overline{L}_{1j}$ $(2\leq j\leq m_1)$ are algebraically independent over $K_1$. 
Hence, since 
\[
U_s \cong \overline{A}/(0:\prod_{j=2}^{m_1} l_{1j}^{d_{1j}})
\cong K[x_2,\ldots,x_n]/
((\prod_{j=1}^{m_2}\overline{L}_{2j}^{d_{2j}}, \ldots, \prod_{j=1}^{m_n}\overline{L}_{nj}^{d_{nj}})
:\prod_{j=2}^{m_1}\overline{L}_{1j}^{d_{1j}}), 
\]
it follows from Lemma~\ref{second key lemma} that $U_s$ has the SLP

We note that 
it is possible to regard $U_1,\ldots,U_{s-1}$ 
as the full set of the central simple modules of 
$(B=R/I+(\prod_{j=2}^{m_1} L_{1j}^{d_{1j}}), \overline{z})$, 
where $\overline{z}$ is the image of $z$ in $B$. 
Since 
\[
B \cong 
K[x_1, x_2, \ldots, x_n]/( \prod_{j=2}^{m_1}L_{1j}^{d_{1j}}, \prod_{j=1}^{m_2}L_{2j}^{d_{2j}}, \ldots, \prod_{j=1}^{m_n}L_{nj}^{d_{nj}} ), 
\]
we have that $B$ has the SLP by the induction hypothesis on $m_1$. 
Let $K_2$ be the subfield of $K$ 
generated by the set $\Xi_1 \cup \Xi_2$, 
and note that 
the coefficients $\{\xi_1^{(1,1)}, \xi_2^{(1,1)},\ldots,\xi_n^{(1,1)}\}$ 
of $L_{11}$ are algebraically independent over $K_2$.  
Then $\overline{z}$ is an SL element of $B$ by Lemma~\ref{key lemma}, 
and hence $U_1,\ldots,U_{s-1}$ have the SLP 
by the proof (i) $\Rightarrow$ (ii) of Theorem 1.2 in \cite{HW1}. 
This completes the proof. 
\end{proof}



%

\end{document}